\documentclass[a4paper,11pt,reqno,draft]{amsart}
\usepackage{amsmath, amssymb, amsthm, amsfonts, amsgen, stmaryrd, mathrsfs}
\usepackage[centertags]{amsmath}
\usepackage{indentfirst}
\usepackage{fancyhdr}
\usepackage{enumitem}
\usepackage{color}
\numberwithin{equation}{section}
\newcommand{\ds}{\displaystyle}
\newcommand{\R}{{\mathbb{R}}}

\newcommand{\N}{{\mathbb{N}}}
\newcommand{\dx}{\,dx}
\newcommand{\dt}{\,dt}

\newcommand{\ie}{{; \it i.e., }}

\newcommand{\wto}{\rightharpoonup}
\newcommand{\dist}{{\rm dist}}

\newcommand{\E}{\mathcal{E}}

\newcommand{\AAA}{\mathbb{A}}
\newcommand{\MM}{\mathbb{M}}
\newcommand{\sym}{\mathrm{sym}}

\let\e=\varepsilon
\let\d=\delta 
\let\O=\Omega

\let\G=\Gamma

\newcommand{\uxiy}{u^\xi_y}
\newcommand{\pixi}{\pi_\xi}
\newcommand{\Ho}{{\mathcal H}^0}
\newcommand{\Sn}{\mathbb{S}^{n-1}}

\newcommand{\RRR}{\color{red}}

\newcommand{\EEE}{\color{black}}

\newtheorem{definition}{Definition}[section]
\newtheorem{lemma}[definition]{Lemma}
\newtheorem{theorem}[definition]{Theorem}
\newtheorem{proposition}[definition]{Proposition}

\newtheorem{remark}[definition]{Remark}

\begin{document}
\title[Phase-field approximation and piecewise-rigid maps]
{Phase-field approximation of functionals defined on piecewise-rigid maps}

\author[M. Cicalese]{Marco Cicalese}
\address[M. Cicalese]{Zentrum Mathematik - M7, Technische Universit\"at M\"unchen, Boltzmannstrasse 3, 85747 Garching, Germany}
\email{cicalese@ma.tum.de}

\author[M. Focardi]{Matteo Focardi} 
\address[M. Focardi]{DiMaI ``U. Dini'', Universit\`a di Firenze, V.le G.B. Morgagni 67/A, 50134 Firenze, Italia}
\email{matteo.focardi@unifi.it}

\author[C.I. Zeppieri]{Caterina Ida Zeppieri}
\address[C.I. Zeppieri]{Angewandte Mathematik, Universit\"at M\"unster, Einsteinstr. 62, 48149 M\"unster, Germany}
\email{caterina.zeppieri@uni-muenster.de}

\begin{abstract}
We provide a variational approximation of Ambrosio-Tortorelli type for brittle fracture energies of piecewise-rigid solids.
Our result covers both the case of geometrically nonlinear elasticity and that of linearised elasticity.
\end{abstract}

\maketitle

{\small{ 
 \noindent {\bf Keywords}: phase-field models, elliptic approximation, free-discontinuity problems, $\G$-convergence, geometric rigidity, piecewise-rigid maps, linearised elasticity, brittle fracture. 

\vspace{6pt} \noindent {\it 2000 Mathematics Subject
Classification:} 49J45,
49Q20,  
74B20, 74G65.
}}


\section{Introduction}

According to the Griffith theory of crack-propagation in brittle materials \cite{G20}, the equilibrium configuration of a fractured body is determined by balancing the reduction in bulk elastic energy $\E^{e}$ stored in the material with the increment in fracture energy $\E^{f}$ due to the formation of a new free surface.  For those materials for which crack-growth can be seen as a quasi-static process, the  equilibrium  configurations are obtained, at each time, by solving a minimisation problem involving the total free energy of the system\ie $\E:=\E^{e}+\,\E^{f}$.   

For hyperelastic brittle materials a prototypical elastic energy $\E^{e}$ is of the form 
\begin{equation}\label{E-el}
\E^{e}(u,K)=\mu\int_{\Omega\setminus K}W(\nabla u)\, dx,
\end{equation}
where $\Omega\subset\R^{3}$ is open, bounded and represents the reference configuration of a body which is fractured along a sufficiently regular closed surface $K\subset\O$, and $u:\Omega\setminus K\to\R^{3}$ is the deformation map, which is smooth outside $K$. 
In \eqref{E-el} the constant $\mu>0$ represents the shear modulus of the material and $W:\MM^{{3\times 3}}\to [0,+\infty)$ the stored elastic energy density. In the setting of nonlinear elasticity $W$ is assumed to be frame indifferent and to vanish only on $SO(3)$, the set of $3\times3$ rotation matrices; moreover, close to $SO(3)$ the function $W(\cdot)$ behaves like $\dist^{2}(\cdot,SO(3))$. 
 
In a simplified isotropic setting, the fracture energy of a brittle material obeys the Griffith criterion and is proportional to the area of the crack-surface $K$\ie
\begin{equation}\label{E-f}
\E^{f}(K)=\gamma\,{\mathcal H}^{2}(K),
\end{equation}
where the proportionality constant $\gamma>0$ measures the fracture toughness (or fracture resistance) of the material. 

Choosing $\E^e$ and $\E^f$ as in \eqref{E-el} and \eqref{E-f}, respectively, the total energy $\E$ takes the form
\[
\E(u,K)=\mu\int_{\Omega\setminus K}W(\nabla u)\, dx+\gamma\,{\mathcal H}^{2}(K).
\]
The functional $\E$ can then be minimised by resorting to a weak formulation of the problem in De Giorgi and Ambrosio's space of \emph{special functions of bounded variation} $SBV(\O)$ \cite{DGA}. In this way the pair $(u,K)$ is replaced by a single variable $u$ which can be discontinuous on a lower-dimensional set $J_u$, which now plays the role of the crack-surface $K$. Moreover if $u \in SBV(\O)$ the distributional derivative $Du$ can be decomposed into a volume contribution $\nabla u$, which is to be interpreted as the deformation gradient outside the crack, and a surface contribution concentrated along the crack-set $J_u$. In the $SBV(\O)$-setting the energy $\E$ then becomes     
\begin{equation}\label{intro:E}
\mathcal F(u)=\mu\int_{\Omega}W(\nabla u)\, dx+\gamma\,{\mathcal H}^{2}(J_{u}),
\end{equation}
and its minimisation can be carried out by applying the direct methods to $\mathcal F$, or to its relaxed functional (cf. \cite{AFP00}).
  
Functionals as in \eqref{intro:E} are commonly referred to as \emph{free-discontinuity functionals} and play a central role both in fracture mechanics \cite{FM, BouFraMa,BouFraMa-1} and in computer vision \cite{MS} and have been extensively studied in the last decades \cite{AFP00, Braides-1}. 

In this paper we are interested in the case when the material parameters in \eqref{intro:E} satisfy the relation $\mu/\gamma\gg1$, or, up to a renormalisation, when $\mu \gg 1$ and $\gamma =O(1)$. 
This parameter-regime is typical of rigid solids\ie of solids which deform without storing any elastic energy. In fact, being $\mu$ large (and $\gamma =O(1)$), a deformation $u$ shall satisfy $W(\nabla u)=0$ which is equivalent to asking $\nabla u \in SO(3)$ almost everywhere in $\O$. However, since $u\in SBV(\O)$, the differential constraint $\nabla u \in SO(3)$ does not prevent $u$ to jump and thus fracture to occur. Hence, for $\mu \gg1$ the energy-functional $\mathcal F$ models those brittle solids which exhibit a rigid behaviour in a number of subregions of $\O$ which are separated from one another by a discontinuity surface. Mathematically, these configurations are described by the so-called \emph{piecewise-rigid} maps on $\O$ and are denoted by $PR(\O)$. Namely, $u\in PR(\O)$ if
\begin{equation}\label{intro:C-A}
u(x)=\sum_{i\in \N} (\AAA_i x+ b_i)\chi_{E_i}(x),
\end{equation}
where, for every $i \in\N$, $\AAA_i\in SO(3)$, $b_{i}\in\R^{3}$, and $(E_i)$ is a Caccioppoli partition of $\O$. 
Then, up to a lower-oder bulk-contribution, in the regime $\mu\gg1$ the total energy $\E$ of a brittle rigid solid can be identified with its fracture energy; the latter coincides with the surface term in \eqref{intro:E} where now the deformation-variable $u$ belongs to the space $PR(\O)$ (see \cite{FrS}). 

Despite their simple analytical expression, energy functionals of type $\gamma \, \mathcal H^{n-1}(J_u)$ are notoriously difficult to be treated numerically, due to their explicit dependence on the discontinuity surface $J_u$. To develop efficient methods to compute their energy minimisers and to analyse phenomena like crack-initiation, crack-branching or arrest in nonlinearly elastic brittle materials, in the engineering community suitable ``regularisations'' have been recently proposed, where the surface $J_u$ is replaced by an additional phase-field variable $v\in [0,1]$ (see \textit{e.g.}, \cite{CK14, Wu18} and references therein). In these models the phase-filed  variable $v$ interpolates between the sound state (corresponding to $v=1$) and the fractured state of the material (corresponding to $v=0$) and it is to be interpreted as a damage variable in the spirit of \cite{PM0,PM1, PAMM}.

It is well-known that the relationship between variational models for brittle fracture and (gradient) damage models can be made rigorous building upon the seminal approximation result of Ambrosio and Tortorelli \cite{AT90, AT92} as shown in \cite{focardi, Cham, BBZ, BCR19, BMZ21, Vic, Iur14, ChCr19}, just to mention a few examples. Furthermore, damage models \'a la Ambrosio-Tortorelli can be also used to approximate fracture models of cohesive-type as shown, \textit{e.g.}, in \cite{ABS, Al-Fo, CFI16, DMI13, Fo-Iur, Iur13, CVG, Cr19}.     

The purpose of the present paper is to establish a rigorous mathematical connection between damage models of Ambrosio-Tortorelli type and variational fracture models for brittle piecewise-rigid solids.   
In other words, in this work we provide an elliptic approximation of functionals of type 
\begin{equation}\label{i:F}
F(u) = \gamma \, \mathcal H^{n-1}(J_u), \quad u\in PR(\O),
\end{equation}
where the constraint $u\in PR(\O)$ is reminiscent of the nonlinear elastic energy density $W$, which satisfy $W^{-1}(\{0\})=SO(n)$.   

Namely, we show that for $(u,v)\in W^{1,2}(\O;\R^n)\times W^{1,2}(\O)$, $0\leq v \leq 1$ the family of functionals    
\begin{equation}\label{intro:AT}
F_{\e}(u,v)= \int_{\O} k_\e\, v^2\,W(\nabla u)\dx +\frac{\gamma}{2} \int_\O \bigg(\frac{(v-1)^2}{\e}
+\e|\nabla v|^2\bigg)\dx,
\end{equation}
converges, in the sense of De Giorgi's $\Gamma$-convergence \cite{DalMaso,Braides}, to the functional \eqref{i:F}, under the assumption that $k_\e \to +\infty$, as $\e\to 0$.
As in the case of the Modica-Mortola functional \cite{M87, MM77} and of the Ambrosio-Tortorelli approximation \cite{AT90, AT92}, in \eqref{intro:AT} the singular-perturbation parameter $\e>0$ determines the thickness of the diffuse interface around the limit discontinuity surface $J_u$, while the diverging parameter $k_{\e}$ is proportional to the stiffness of the material, hence to the constant $\mu$ appearing in \eqref{intro:E}. 
More precisely, in Theorem \ref{thm:G-conv} we prove that if the zero-set of the bulk energy density
$W$ coincides with $SO(n)$ and for every $\AAA\in \mathbb M^{n\times n}$ it holds
\[
W(\AAA)\geq \alpha \, \dist^2(\AAA,SO(n)),
\]
for some $\alpha >0$, then the family $(F_\e)$ $\Gamma$-converges to $F$, in the $L^1(\O;\R^n)\times L^1(\O)$ topology.  
The proof of Theorem \ref{thm:G-conv} takes advantage of a number of analytical tools. 
First, to determine the set of the limit deformations we use a piecewise-rigidity result in $SBV(\O)$ by Chambolle, Giacomini and Ponsiglione \cite{CGP}
(cf.\ Theorem \ref{t:CGP}). The latter is the counterpart of the Liouville rigidity Theorem for deformations of brittle elastic materials and provides a characterisation of discontinuous deformations with zero elastic energy as a collection of an at most countable family of rigid motions defined on an underlying Caccioppoli partition of $\O$. 
To match the assumptions of Chambolle, Giacomini, and Ponsiglione's result we use a global argument of Ambrosio \cite{Amb} which is based on the co-area formula and is tailor-made to gain compactness in $SBV$. Namely, starting from a pair $(u_\e,v_\e) \subset W^{1,2}(\O,\R^n)\times W^{1,2}(\O)$ with equi-bounded energy $F_\e$ we use the co-area formula to find a suitable sublevel set of $v_\e$ in which $u_\e$ can be modified to obtain a new sequence $(\tilde u_\e)\subset SBV(\O)$ which differs from $u_\e$ on a set of vanishing measure and moreover satisfies 
\[
\sup_{\e>0} \bigg(k_\e \int_\O \dist^2(\nabla \tilde u_\e, SO(n))\dx + \mathcal H^{n-1}(J_{\tilde u_\e})\bigg)<+\infty. 
\]
The estimate above, combined with a result of Zhang which guarantees that the zero-set of the quasiconvexification of ${\rm dist}(\cdot,SO(n))$ coincides with $SO(n)$, proves that any $L^{1}$ limit $u$ of $\tilde u_{\e}$ satisfies $\nabla u\in SO(n)$ a.e. in $\Omega$. Eventually, the Chambolle-Giacomini-Ponsiglione piecewise-rigidity Theorem yields that $u$ is a piecewise-rigid map. 
The construction of Ambrosio additionally provides us with the sharp lower bound. Indeed, the perimeter of the sublevel sets of $v_\e$ chosen as above prove to be asymptotically larger than the interfacial energy-contribution of the piecewise rigid limit deformation. 
As in the case of the Ambrosio-Tortorelli functional, the sharp interfacial energy is defined in terms of a one-dimensional optimal profile problem. 
The upper bound is then proven first by resorting to a density argument and then by an explicit construction. Namely, we use the density in $PR(\O)$ of \emph{finite} partitions subordinated to Caccioppoli sets which are \emph{polyhedral} \cite{BCG}. Then, for these partitions, a recovery sequence matching asymptotically the sharp lower bound can be contructed by creating a layer of order $\e$ around the jump set of the target function $u$, in which the transition is one-dimensional and is obtained by a suitable scaling of the optimal profile. 

As for the Ambrosio-Tortorelli approximation of the Mumford-Shah functionals (see also, \text{e.g.}, \cite{focardi, BBZ, BCR19, BMZ21, Vic}) also in our case the regularised bulk and surface energy in \eqref{intro:AT} separately converge to their sharp counterparts. Namely, in this case the bulk term in \eqref{intro:AT} vanishes in the limit due to the presence of the diverging parameter $k_\e$, that is, equivalently, limit deformations have (approximate) gradients in $SO(n)$ a.e. in  $\Omega$.
Similarly, the Modica-Mortola term in \eqref{intro:AT} approximates the limit surface energy, which in our model carries the whole energy contribution.  

It is worth mentioning that the arguments in Theorem \ref{thm:G-conv} can be extended (resorting to by-now standard modifications)
to cover the case of anisotropic surface-integrals which model the presence of preferred cleavage planes in single crystals (cf.\ Remark \ref{r:inhomogeneous anisotropic}). 

In Theorem \ref{thm:G-conv-K} we generalise the approximation result Theorem \ref{thm:G-conv} to the case of energy densities $W$ vanishing on a compact set $\mathcal K \subset \mathbb M^{n\times n}$ for which a piecewise-rigidity result analogous to the one for $SO(n)$-valued discontinuous deformations holds true. In fact in \cite{CGP} piecewise rigidity is proven, more in general, for those $\mathcal K$ for which a quantitative $L^p$-rigidity estimate holds  (see Section \ref{s:generalizations} for more details).  
In this way, multiple incompatible wells can be also taken into account. From a mechanical point of view, the incompatibility describes those solids for which no fine-scale phase-mixtures are allowed in solid-solid transformations (see \cite{MLN}).
A list of non trivial examples of possible compact sets $\mathcal K$ fulfilling the assumptions of Theorem \ref{thm:G-conv-K} is also included.

Finally, in Theorem \ref{thm:G-conv-lin} a further approximation result is provided, which covers the case of linearised elasticity.

\section{Notation and preliminaries}

\subsection{Notation}
In what follows $\O\subset \R^n$ denotes a bounded domain (\textit{i.e.}, an open and connected set) with Lipschitz boundary.
We use a standard notation for Lebesgue and Sobolev spaces, and for the Hausdorff measure. The Euclidean 
scalar product in $\R^n$ is denoted by $\langle \cdot,\cdot\rangle$.

We refer the reader to the book \cite{AFP00} for a comprehensive introduction to the theory of 
functions of bounded variation  and of (generalised) special functions of bounded variation 
$(G)SBV(\O)$ and to \cite{DM} for the definition and main properties of generalised special functions of bounded deformation 
$GSBD(\O)$. In any of these cases we shall deal with the proper subspaces of these functional spaces in $L^1(\O,\R^n)$. 


Below we briefly recall the notation and the main results for one-dimensional sections of $GSBV$ functions, in a form that we need, since we will make an extensive use of these. 

\subsection{Slicing} 
Let $n\in \N$; for fixed $\xi\in \mathbb{S}^{n-1}:=\{\xi\in\mathbb{R}^n:|\xi|=1\}$, 
let $\pi_\xi$ be the orthogonal projection onto the hyperplane 
$\Pi^\xi:=\big\{y\in\mathbb{R}^n:\,y\cdot\xi=0\big\}$, and for every subset $U\subset\R^n$ set 
\[
U_y^\xi:=\big\{t\in\mathbb{R}:y+t\xi\in U\big\},\quad \text{ for $y\in \Pi^\xi$}.
\]
Let $w:\Omega\to\mathbb{R}^n$, then define the slices $w_y^\xi:\Omega_y^\xi\to\mathbb{R}^n$ by
\begin{equation}\label{slice}
w_y^\xi(t):=w(y+t\xi)\,.
\end{equation} 
We recall the slicing theorem in $GSBV$ 
(see \cite[Section~3.11, Proposition~4.35]{AFP00}). 
\begin{theorem}\label{slicing}
Let $w\in L^1(\Omega,\R^n)$, and let $\{\xi_1,...,\xi_n\}$ be an orthonormal basis of $\R^n$. 
Then the following two conditions are equivalent:
\begin{itemize}

\item[(i)] For every $1\leq i \leq n$, 
$w_y^\xi\in GSBV(\Omega^\xi_y,\R^n)$ for  $\mathcal{H}^{n-1}$-a.e. $y \in \Pi^\xi$
and
\[
\int_{\Pi^\xi}\big|D\big((w_y^\xi)_T\big)\big|(\Omega^\xi_y)\,d\mathcal{H}^{n-1}(y)<+\infty,
\]
where $\langle(w_y^\xi)_T,{\mathrm e}_i\rangle=(\langle w_y^\xi,{\mathrm e}_i\rangle\wedge T)\vee(-T)$, with $T>0$ and $i\in\{1,\ldots,n\}$, 
where $\{{\mathrm e}_1,\ldots{\mathrm e}_n\}$ is the canonical base of $\R^n$;
\item[(ii)] $w\in \big(GSBV(\Omega)\big)^n$.
\end{itemize}

Moreover, if $w\in \big(GSBV(\Omega)\big)^n$ and $\xi \in \R^n\setminus\{0\}$ the following properties hold:
\begin{itemize}

\item[(a)] $(w_y^\xi)'(t)=\nabla w\left(y+t\xi\right)\xi$ 
for $\mathcal{L}^1$-a.e. $t\in \Omega^\xi_y$ and for 
$\mathcal{H}^{n-1}$-a.e. $y\in\Pi^\xi$;

\item[(b)] $J_{w_y^\xi}=\big(J_w^\xi\big)^\xi_y$ for 
$\mathcal{H}^{n-1}$-a.e. $y\in\Pi^\xi$, 
where
\[
J_w^\xi:=\{x\in J_w:\,[w](x)\cdot\xi\neq 0\};
\] 

\item[(c)] if $\mathcal{H}^{n-1}(J_w)<+\infty$, for 
$\mathcal{H}^{n-1}$-a.e. $\xi \in \mathbb{S}^{n-1}$
\begin{equation}\label{e:juxi}
\mathcal{H}^{n-1}(J_w\setminus J_w^\xi)=0. 
\end{equation}
\end{itemize}
\end{theorem}
We observe that (c) simply follows from a Fubini type argument noting that 
\[
\mathcal{H}^{n-1}(\{\xi\in{\mathbb S}^{n-1}: [w](x)\cdot\xi=0\})=0,
\]
for every $x\in J_w$, where $[w](x)$ is the difference of the one-sided traces of $w$ at $x\in J_u$.

We also note that, if $w_k,w\in L^1(\Omega,\mathbb{R}^n)$ and $w_k\to w$ in $L^1(\Omega,\mathbb{R}^n)$, 
then for every $\xi\in \mathbb{S}^{n-1}$ there exists a subsequence $(w_{k_j})$ of $(w_k)$ such that
\[
(w_{k_j})^\xi_y\to v_y^\xi\; \text{ in }\; L^1(\Omega^\xi_y,\R^n)\quad
\textrm{for $\mathcal{H}^{n-1}$-a.e.\ $y\in\pixi(\Omega)$}.
\]
\subsection{Caccioppoli-affine functions} We recall here the definition of Caccioppoli-affine and piecewise-rigid function. Moreover we also recall the piecewise-rigidity result \cite[Theorem 1.1]{CGP} in a variant which is useful for our purposes. 
\begin{definition}
A map $u \colon \O \to \R^{n}$ is called Caccioppoli-affine if there exist matrices $\AAA_i\in \MM^{n\times n}$ and vectors $b_i\in\R^{{n}}$ such that 
\begin{equation}\label{C-a}
u(x)=\sum_{i\in \N} (\AAA_i x+ b_i)\chi_{E_i}(x),
\end{equation}
with $(E_i)$ Caccioppoli partition of $\O$. In particular, if 
$\AAA_i\in SO(n)$ for every $i \in\N$, 
then $u$ as in \eqref{C-a} is called piecewise rigid. The set of piecewise-rigid functions on $\O$ will be denoted by $PR(\O)$. 
\end{definition}

The measure theoretic properties of Caccioppoli-affine functions are collected in the result below (cf. \cite[Theorem 2.2]{CFI}). 
\begin{theorem}\label{t:CFI}
Let $\O\subset \R^n$ be open, bounded, and with Lipschitz boundary. Let $u \colon \O \to \R^{{n}}$ be Caccioppoli-affine, 
then $u\in (GSBV(\O))^{{n}}$. Moreover, 
\begin{enumerate}
\item $\nabla u = \AAA_i$ $\mathcal L^n\hbox{-}\text{a.e.\ on }\, E_i$, for every $i\in \N$;
\item $J_u = \bigcup_{i\in \N} \partial^*E_i \cap \O$,
up to a set of zero $\mathcal H^{n-1}$-measure.  
\end{enumerate}
\end{theorem}
 
Below we recall a slight generalisation of the piecewise-rigidity result by Chambolle, Giacomini, and Ponsiglione \cite[Theorem 1.1]{CGP}
originally stated in the $SBV$-setting. 

\begin{theorem}\label{t:CGP}
Let $u\in GSBV(\O,\R^n)$ be such that 
$\mathcal{H}^{n-1}(J_u)<+\infty$ and $\nabla u \in SO(n)$ a.e. in $\O$. Then, $u \in PR(\O)$. 
\end{theorem}

\section{Setting of the problem and main result}\label{s:main}
In this section we introduce a family of functionals of Ambrosio-Tortorelli type (cf. \cite{AT90,AT92}) 
and we prove that this family converges to a surface functional of perimeter type which is finite only on piecewise-rigid maps. 

Let $W \colon \O \times \MM^{n\times n} \to [0,+\infty)$ be a Borel function such that $W(x,\AAA)=0$ for every $\AAA\in SO(n)$. Assume moreover that for every $x\in \O$ and every $\AAA\in \MM^{n\times n}$ it holds
\begin{equation}\label{e:lb-W}
W(x,\AAA)\geq \alpha \, \dist^2(\AAA,SO(n)),
\end{equation}
for some $\alpha >0$.

Let $\Phi \colon [0,1] \to [0,1]$ be an increasing and lower semicontinuous function such that $\Phi(0)=0$, $\Phi(1)=1$, $\Phi(t)>0$ for $t>0$; let moreover $V \colon [0,1] \to [0,+\infty)$ be a continuous function with $V^{-1}(\{0\})=\{1\}$. 
For $\e>0$ let $k_\e \to +\infty$, as $\e\to 0$. We consider the phase-field functionals $F_\e \colon L^1(\O,\R^n)\times L^1(\O) \longrightarrow [0,+\infty]$ defined as 
\begin{equation}\label{Fe}
F_\e(u,v):=\begin{cases}
\ds\int_{\O}\bigg(k_\e \Phi(v)\,W(x,\nabla u)+\frac{V(v)}{\e}
+\e|\nabla v|^2\bigg)\dx & (u,v)\in W^{1,2}(\O,\R^n)\times W^{1,2}(\O),
\cr 
 & 0\leq v\leq 1 \; \text{ a.e. on $\O$,}\; 
\cr
+\infty & \text{otherwise.}
\end{cases}
\end{equation}
In the following proposition we show that the $\G$-limit of $(F_\e)$ (if it exists) is finite only on the set of piecewise-rigid maps.
\begin{proposition}[Domain of the $\G$-limit]\label{p:domain}
 Let $(u_\e,v_\e) \subset W^{1,2}(\O,\R^n)\times W^{1,2}(\O)$,  $0\leq v_\e \leq 1$ a.e.\ in $\O$, be such that 
 $$
 \liminf_{\e \to 0}F_\e(u_\e,v_\e)<+\infty 
 \quad \text{and}\quad u_\e \to u \; \text{ in } \; L^1(\O,\R^n).
 $$
 Then, $v_\e \to 1$ in $L^1(\O)$ and $u \in PR(\O)$.  
 \end{proposition}
 \begin{proof}
Let $(u_\e,v_\e) \subset W^{1,2}(\O,\R^n)\times W^{1,2}(\O)$,  $0\leq v_\e \leq 1$ a.e.\ in $\O$, be such that 
 $$
 \liminf_{\e \to 0}F_\e(u_\e,v_\e)<+\infty 
 \quad \text{and}\quad u_\e \to u \; \text{ in } \; L^1(\O,\R^n).
$$
We first note that up to subsequences (not relabelled) we have
\begin{equation}\label{en-bd}
\sup_{\e}F_\e(u_\e,v_\e)<+\infty,
\end{equation}
from which the convergence $v_\e \to 1$ in $L^1(\O)$ easily follows. In fact, for $\eta>0$ we have
\begin{equation}\label{e:conv-meas}
\mathcal L^n(\{ 1- \eta > v_\e\}) \min \{V(s) \colon s \in [0, 1-\eta)\} \leq \int_\O V(v_\e)\dx \leq \e \sup_{\e}F_\e(u_\e,v_\e).
\end{equation}
Since $V^{-1}(\{0\})=1$, the minimum in the left hand side of \eqref{e:conv-meas} is strictly positive. 
Therefore, gathering \eqref{e:conv-meas} and \eqref{en-bd} implies that $v_\e \to 1$ in measure.  
The latter, together with the uniform bound satisfied by $(v_{\e})$ immediately gives $v_\e \to 1$ in $L^1(\O)$. 

We are then left to show that $u \in PR(\O)$. 
To do so, we resort to a global technique introduced by Ambrosio in \cite{Amb} (see also \cite{focardi, focardi_tesi}). 
That is, starting from $(u_\e,v_\e) \subset W^{1,2}(\O,\R^n)\times W^{1,2}(\O)$ as in \eqref{en-bd} we construct a sequence $(\tilde u_\e) \subset GSBV(\O,\R^n)$ satisfying the two following properties: 
\begin{equation}\label{tilde-u}
\tilde u_\e \to u\; \text{ in } \; L^1(\O,\R^n) \quad \text{and}\quad \sup_\e \bigg(\int_{\O}|\nabla \tilde u_\e|^2\dx + \mathcal{H}^{n-1}(J_{\tilde u_\e}) \bigg)<+\infty.
\end{equation}
The Cauchy-Schwartz Inequality and the $BV$ Coarea Formula \cite[Theorem~3.40]{AFP00} yield
\begin{eqnarray*}
F_{\e}(u_\e,v_\e) &\geq & 2\int_\O \sqrt{V(v_\e)}\,|\nabla v_\e|\dx
\\
&\geq& 2
\int_0^1\sqrt{V(s)}\,\mathcal H^{n-1}(\partial^*\{v_\e <s\})\,ds.
\end{eqnarray*}
%
Thus, for every $\delta \in (0,\frac12)$ fixed, the Mean-value Theorem provides us with $\lambda_\e^\delta \in (\delta,1-\delta)$ 
such that
\begin{align}\label{super-level-set}
F_\e(u_\e,v_\e) &\geq  2 \int_\delta^{1-\delta} \sqrt{V(s)}\,\mathcal H^{n-1}(\partial^*\{v_\e <s\})\,ds 
\nonumber\\ 
&\geq 2 \bigg(\int_\delta^{1-\delta} \sqrt{V(s)}\,ds\bigg)\, 
\mathcal H^{n-1}(\partial^*\O^\d_\e), 
\end{align}
where we have set $\O^\delta_\e:=\{x\in \O \colon v_\e(x) < \lambda_\e^\delta\}$. We notice that gathering \eqref{en-bd}-\eqref{super-level-set} we obtain that $\O^\delta_\e$ is a set of finite perimeter.

Now let $\AAA\in SO(n)$ and set
$$
\tilde u_\e := u_\e\, \chi_{\O \setminus \O^\delta_\e} + \AAA x\, \chi_{\O^\delta_\e}; 
$$ 
since $u_\e \in W^{1,2}(\O,\R^n)$ and $\O^\delta_\e$ is a set of finite perimeter, then $\tilde u_\e \in {GSBV(\Omega,\R^n) \cap L^1(\O,\R^n)}$.  Moreover the approximate gradient of $\tilde u_\e$ is given by
\begin{equation}\label{e:der ueps}
\nabla \tilde u_\e=\nabla u_\e\chi_{\O \setminus \O^\delta_\e} + 
\AAA\, \chi_{\O^\delta_\e}
\end{equation}
and $J_{\tilde u_\e} \subset \partial^*\O^\d_\e$; therefore \eqref{en-bd} and \eqref{super-level-set} 
give 
\begin{equation}\label{e:salto tildeueps}
\sup_\e \mathcal H^{n-1} (J_{\tilde u_\e})<+\infty.
\end{equation}
We also notice that since $v_\e \to 1$ in $L^1(\O)$ then $\mathcal L^n(\O^\delta_\e) \to 0$, as $\e \to 0$, and thus
$\tilde u_\e \to u$ in $L^1(\O,\R^n)$. 

On the other hand, by \eqref{e:lb-W}, by definition of $\tilde u_\e$, and by \eqref{e:der ueps} we also have 
\begin{eqnarray}\label{dist-to-0}
\nonumber
F_\e(u_\e,v_\e) &\geq&  \alpha k_\e
\int_{\O} \Phi(v_\e)\,\dist^2(\nabla u_\e,SO(n))\dx 
\\ \nonumber
&\geq & \alpha k_\e  \Phi(\delta)
\int_{\O \setminus \O^\delta_\e}  \dist^2(\nabla u_\e,SO(n))\dx 
\\ 
&= & \alpha k_\e  \Phi(\delta)
\int_{\O} \, \dist^2(\nabla \tilde u_\e,SO(n))\dx  
\end{eqnarray}
from which we immediately deduce that 
\[
\sup_\e \|\nabla \tilde u_\e\|_{L^2(\O,\R^{n\times n})}<+\infty
\] 
and hence 
$\tilde u_\e$ satisfies \eqref{tilde-u}, as desired. Then, by Ambrosio's $GSBV$ Closure Theorem \cite[Theorem~4.36]{AFP00} we deduce 
both that $u \in GSBV(\O,\R^n)$ and
\begin{equation}\label{e:tildeueps prop}
\nabla \tilde u_\e \wto \nabla u \; \text{ in $L^2(\O,\R^{n \times n})$ \; and } \; 
\mathcal H^{n-1}(J_u)\leq\liminf_{\e\to 0} \mathcal H^{n-1}(J_{\tilde u_\e})<+\infty\,.
\end{equation}
By virtue of \eqref{dist-to-0} we also have
\begin{equation}\label{e:piecewise rigidity of u}
\int_\O \dist^2(\nabla \tilde u_\e,SO(n))\dx \leq \frac{F_\e(u_\e,v_\e) } {\alpha k_\e \Phi(\delta)} \to 0\; \text{ as }\; \e \to 0,
\end{equation}
thus by the lower-semicontinuity of functionals defined in $GSBV$ with respect to the weak $L^1$ convergence of approximate gradients  
(cf. \cite[Theorem 5.29]{AFP00} and \cite[Theorem~1.2]{KR}) we deduce that
\begin{equation*}
\int_{\O} Q(\dist^2(\cdot,SO(n))(\nabla u)\dx 
\leq \liminf_{\e \to 0} \int_{\O} \dist^2(\nabla \tilde u_\e,SO(n))\dx =0\,,
\end{equation*}
where $Q(\dist^2(\cdot,SO(n))(\AAA)$ denotes the quasi-convex envelope of $\dist^2(\cdot,SO(n))$ computed at $\AAA\in \mathbb M^{n\times n}$ (cf. \cite[Section~5.3]{G}). 
Since by \cite[Theorem~1.1]{ZH97} 
(see also \cite[Theorem~1.1]{Zh92}) we have
$$
\big\{\AAA\in \MM^{n\times n} \colon Q(\dist^2(\cdot,SO(n))(\AAA)=0\big\}=SO(n),
$$
we immediately obtain that $\nabla u \in SO(n)$ a.e.\,in $\O$. Eventually, appealing to Theorem~\ref{t:CGP} we deduce that $u \in PR(\O)$, and hence the claim. 
\end{proof}
\begin{remark}\label{r:strong L2 conv}
{\rm
We observe that the family $(\tilde u_\e)$ exhibited in Proposition~\ref{p:domain} 
actually satisfies $\nabla \tilde u_\e \to \nabla u$ in $L^2(\O,\MM^{n\times n})$. 
Indeed, combining \eqref{en-bd} and \eqref{dist-to-0} gives 
$$
 \int_{\O}\big||\nabla \tilde u_\e|-\sqrt n\big|^2\dx \leq \frac{F_\e(u_\e,v_\e)}{k_\e \, \alpha \,\Phi(\delta)} \to 0\; \text{ as }\; \e \to 0
$$
and therefore 
$$
\lim_{\e \to 0}\int_{\O} |\nabla \tilde u_\e|^2\dx= n \mathcal L^n(\O) = \int_{\O} |\nabla u|^2\dx.
$$
The latter combined with the weak convergence $\nabla \tilde u_\e \wto \nabla u$ in $L^2(\O,\MM^{n\times n})$ 
yields the claim.

} 
\end{remark}

The next theorem establishes a $\Gamma$-convergence result for the functionals $F_\e$.

\begin{theorem}
\label{thm:G-conv} 
The family of functionals $(F_\e)$ defined in \eqref{Fe} $\Gamma(L^1(\O,\R^n)\times L^1(\O))$-converges 
to the functional $F \colon L^1(\O,\R^n)\times L^1(\O)\longrightarrow [0,+\infty]$ given by
\begin{equation}\label{F}
F(u,v):=\begin{cases}
2C_V \mathcal H^{n-1}(J_u) & \text{if $u \in PR(\O)$\; and }\; v=1\; \text{a.e. in }\; \O,
\cr
+\infty & \text{otherwise},
\end{cases}
\end{equation}
where $C_V:=2\int_0^1 \sqrt{V(s)}\,ds$. 
\end{theorem}

\begin{proof}
We divide the proof into two steps. 

\medskip

\noindent \emph{Step 1: Ansatz-free lower bound.} 
Let $(u,v) \in L^1(\O,\R^n)\times L^1(\O)$ be arbitrary; we need to show that
for all sequences $(u_\e,v_\e) \to (u,v)$ in $L^1(\O,\R^n)\times L^1(\O)$, $0\leq v_\e \leq 1$ a.e.\ in $\O$, we have
\begin{equation}\label{lb}
\liminf_{\e\to 0}F_\e(u_\e,v_\e) \geq F(u,v). 
\end{equation}
Without loss of generality, up to the extraction of a subsequence, we may assume that the liminf in \eqref{lb} is a limit; therefore we have
\begin{equation}\label{bounded}
\sup_\e F_\e(u_\e,v_\e)<+\infty
\end{equation}
Then Proposition \ref{p:domain} 
readily implies that $u \in PR(\O)$ and $v=1$ a.e. in $\O$.  

To prove \eqref{lb} we start noticing that by \eqref{super-level-set} and the Fatou Lemma we have
\[
\liminf_{\e \to 0} F(u_\e,v_\e)\geq 2\int_\delta^{1-\delta}\sqrt{V(s)}\,
\liminf_{\e \to 0}\mathcal{H}^{n-1}(\partial^*\{v_\e>s\})\,ds\,,
\]
then, to conclude it suffices to show that 
\begin{equation}\label{e:surfaceinq}
\liminf_{\e\to 0}\mathcal{H}^{n-1}(\partial^*\{v_\e>s\})
\geq 2\mathcal{H}^{n-1}(J_u)
\end{equation}
for $\mathcal{L}^1$-a.e.\ $s\in(\delta,1-\delta)$, and then let $\delta\to 0^+$.

The estimate in \eqref{e:surfaceinq} can be obtained via slicing similarly as in \cite{focardi, focardi_tesi}. 
Specifically, fix $s\in(\delta,1-\delta)$ for which the left-hand side of 
\eqref{e:surfaceinq} is finite, and set $\O_\e:=\{v_\e<s\}$; we notice that $\mathcal L^n(\O_\e)\to 0$, as $\e\to 0$.
We now claim that for every open subset $U\subset \Omega$ and every
$\xi\in\Sn$ we have
\begin{equation}\label{e:percut}
\liminf_{\e\to 0}\mathcal{H}^{n-1}(J_{\chi_{\O_\e}}\cap U)
\geq 2\int_{\pixi(U)}\Ho(J_{\uxiy}\cap U)d\mathcal{H}^{n-1},
\end{equation}
for $\mathcal{H}^{n-1}$-a.e.\ $y\in \pixi(U)$.
Assume for the moment that \eqref{e:percut} holds true, then the co-area formula for rectifiable sets \cite[Theorem~2.93]{AFP00} 
yield the following lower-semicontinuity estimate
\begin{align}\label{e:percut20}
\liminf_{\e\to 0}\mathcal{H}^{n-1}(\partial^*\{v_\e>s\}\cap U)
&=\liminf_{\e\to 0}\mathcal{H}^{n-1}(J_{\chi_{\O_\e}}\cap U)\notag
\\
&\geq 2\int_{\pixi(U)}\Ho(J_{\uxiy}\cap U)d\mathcal{H}^{n-1}\notag
\\
&=2\int_{J^{\xi}_u\cap U}|\nu_u\cdot\xi|d\mathcal{H}^{n-1}.
\end{align}
Thanks to \eqref{e:juxi} for $\mathcal{H}^{n-1}$-a.e.\ $\xi\in\Sn$ we have $\mathcal{H}^{n-1}(J_u\setminus J_u^\xi)=0$, therefore from \eqref{e:percut20} we also infer
\begin{equation}\label{e:percut2}
\liminf_{\e\to 0}\mathcal{H}^{n-1}(\partial^*\{v_\e>s\}\cap U)
\geq 2\int_{J_u\cap U}|\nu_u\cdot\xi|d\mathcal{H}^{n-1}.
\end{equation}
Then, \eqref{e:surfaceinq} follows from \eqref{e:percut2}
passing to the supremum on a dense sequence $\left(\xi_j\right)$ in ${\mathbb S}^{n-1}$
and invoking \cite[Lemma 2.35]{AFP00}, also noticing that the function
\[
U\mapsto\liminf_{\e\to 0}\mathcal{H}^{n-1}(\partial^*\{v_\e>s\}\cap U)
\] 
is superadditive on pairwise disjoint open subsets of $\O$.

Hence we are now left to prove \eqref{e:percut}. To this end we start observing that for every $\AAA\in\MM^{n\times n}$ it holds
\[
\dist^2(\AAA,SO(n))\geq ||\AAA|-\sqrt{n}|^2\geq\frac12|\AAA|^2-n\,.
\]
Therefore, for every $\AAA\in\MM^{n\times n}$ and every $\xi\in\mathbb{S}^{n-1}$ we get
\begin{equation}\label{e:dist coerc}
\dist^2(\AAA,SO(n))\geq\Big[\frac12|\AAA\xi|^2-n\Big]_+ \,,
\end{equation}
where $[t]_+$ denotes the positive part of $t\in\R$. 

In view of \eqref{e:lb-W}, \eqref{bounded}, \eqref{e:dist coerc} and by the Fatou Lemma we can find a subsequence 
$(u_{\e_j},v_{\e_j})$ of $(u_\e,v_\e)$ 
such that 
\begin{equation}\label{numsal}
\liminf_{\e\to 0}\mathcal{H}^{n-1}(J_{\chi_{\O_\e}}\cap U)=\lim_{j\to+\infty}\mathcal{H}^{n-1}(J_{\chi_{\O_{\e_j}}}\cap U),\\
\end{equation}
and for $\mathcal{H}^{n-1}$-a.e.\ $y\in\pixi(\Omega)$
\begin{equation}\label{e:sliceconv}
\big((u_{\e_j})^\xi_y,(v_{\e_j})^\xi_y\big) \to (\uxiy,1)\;
\text{ in }\; L^1(\Omega^\xi_y,\R^n) \times L^1(\Omega^\xi_y),
\end{equation}
and 
\begin{align}\label{e:fette}
\liminf_{j\to+\infty}\bigg(
\int_{U^\xi_y}&\bigg(\alpha \, k_{\e_j}(v_{\e_j})^\xi_y\Big[\frac12\big(((u_{\e_j})^\xi_y)'\big)^2-n\Big]_+ 
+\frac {V\big((v_{\e_j})^\xi_y\big)}{\e_j}
+\e_j\big(((v_{\e_j})^\xi_y)'\big)^2\bigg)\,dt\notag\\
&+\Ho(J_{\chi_{(\O_{\e_j})^\xi_y}}\cap U)\bigg)\leq c<\infty,
\end{align}
for some constant $c>0$ (which may depend on $y$). 

Let $y\in\pixi(\Omega)$ be fixed and such that both \eqref{e:sliceconv} and \eqref{e:fette} are satisfied; moreover assume that 
\[
\Ho\big(J_{\uxiy}\cap U\big)>0. 
\]
Let $\left\{t_1,...,t_l\right\}$ be an arbitrary subset of $J_{\uxiy}\cap U$, 
and let $(I_i)_{1\leq i\leq l}$ be a family of pairwise disjoint open intervals 
such that $t_i\in I_i$, $I_i\subset \subset U^\xi_y$. 
Then, for every  $1\leq i\leq l$, we have
\begin{equation}\label{e:v-vanish}
s_i:=
\liminf_{j\to+\infty}\big(\inf_{I_i}(v_{\e_j})^\xi_y\big)=0.
\end{equation}
Indeed, if $s_h>0$ for some $h\in\{1,...,l\}$, then
for $j$ sufficiently large we would get
\[
\inf_{I_h}(v_{\e_j})^\xi_y\geq {\frac {s_h}  2}\,,
\]
thus \eqref{e:fette} would give 
\[
\int_{I_h}\big(((u_{\e_j})^\xi_y)'\big)^2dt\leq
2\int_{I_h}\Big[\frac12\big(((u_{\e_j})^\xi_y)'\big)^2-n\Big]_+dt+2n\mathcal{L}^1(I_h)\leq
\frac{2c}{\alpha k_{\e_j}s_h}+2n\mathcal{L}^1(I_h)\,.
\]
Hence, Rellich-Kondrakov's Theorem and \eqref{e:sliceconv} 
would imply that the slice $\uxiy$ belongs to $W^{1,1}(I_h,\R^n)$, thus
contradicting the assumption $\Ho\left(J_{\uxiy}\cap I_h\right)>0$.
Then, thanks to \eqref{e:sliceconv} and \eqref{e:v-vanish} we can find $(r_j^i)\subset I_i$ such that 
\[
\lim_{j\to+\infty} (v_{\e_j})^\xi_y(r_j^i)=0
\]
and $r'_i$, $r''_i\in I_i$, with $r_i<r_j^i<r''_i$ satisfying  
\[
\lim_{j\to+\infty}  (v_{\e_j})^\xi_y\left(r'_i\right) =
\lim_{j\to+\infty}  (v_{\e_j})^\xi_y\left(r''_i\right) =1\,,
\]
which readily gives 
\[
\liminf_{j\to+\infty}\Ho(J_{\chi_{(\O_{\e_j})^\xi_y}}\cap I_i)\geq 2.
\] 
Eventually, the subadditivity of the liminf and the arbitrariness of $l$ yield
\[
\liminf_{j\to+\infty}\Ho(J_{\chi_{(\O_{\e_j}^{})^\xi_y}}\cap U) \geq 2\Ho(J_{\uxiy}\cap U)\,,
\]
so that \eqref{e:percut} follows by integrating the previous inequality on $\pi_\xi(U)$ and using the Fatou Lemma. 
\medskip 

\noindent \emph{Step 2: Existence of a recovery sequence.}  
Let $(u,v) \in L^1(\O,\R^n)\times L^1(\O)$ be arbitrary, in this step we will construct a sequence $(u_\e,v_\e) \to (u,v)$ in $L^{1}(\O,\R^n)\times L^1(\O)$ such that
\begin{equation}\label{ub}
\limsup_{\e\to 0}F_\e(u_\e,v_\e) \leq F(u,v).
\end{equation}
We start by noticing that the inequality in \eqref{ub} is trivial unless we additionally assume that $u \in PR(\O)$ and $v=1$ a.e. in $\O$. Therefore, in particular
we can write $u$ as
\begin{equation}\label{ub:Cp}
u(x)=\sum_{i\in \N}(\AAA_i x+b_i) \chi_{E_i}(x),
\end{equation}
where $\AAA_i \in SO(n)$, $b_i \in \R^n$ for every $i\in \N$, and 
$(E_i)$ is Caccioppoli partition of $\O$. 

By standard density and continuity arguments (cf.\ \cite[Remark 1.29]{Braides}) we notice that it is enough to prove \eqref{ub} in a subset $X$ of $PR(\O)$, which
is dense in $PR(\O)$ in the following sense: for every $u\in PR(\O)$ there exists $(u_j) \subset X$
such that
\begin{equation}\label{density}
u_j \to u\;  \text{ in } \; L^1(\O,\R^n)\; \text{ and }\; \mathcal H^{n-1}(J_{u_j}) \to \mathcal H^{n-1}(J_{u}),   
\end{equation}
for $j \to +\infty$. 

We now claim that $X$ is given by those $u \in PR(\O)$ of the form
\begin{equation}\label{approx}
u(x)=\sum_{i=1}^{N}(\widehat \AAA_i x+\hat b_i) \chi_{\widehat E_i}(x),
\end{equation}
where $\widehat \AAA_i \in SO(n)$, $\hat b_i \in \R^n$, and $\widehat E_i$ is a polyhedral set, for every $i=1,\ldots, N$. 
Indeed given $u$ as in \eqref{ub:Cp} the sequence $(u_N)$ defined as
$$
u_N(x)=\sum_{i=1}^{N-1}(\AAA_i x+ b_i) \chi_{E_i}(x) + 
(\AAA_1x+ b_1) \chi_{\O\setminus \bigcup_{i}^{N-1} E_i},
$$
clearly satisfies $u_N \to u$ in $L^1(\O,\R^n)$, as $N\to +\infty$. Moreover, by lower semicontinuity we have that $\mathcal H^{n-1}(J_u)\leq \liminf_N \mathcal H^{n-1}(J_{u_N})$, 
hence since $\mathcal  H^{n-1}(J_{u_N}) \leq \mathcal H^{n-1}(J_{u})$ for every $N \in \N$, we obtain
$$
\mathcal  H^{n-1}(J_{u_N}) \to \mathcal H^{n-1}(J_{u}),
$$
as $N\to +\infty$.

Further, given the finite partition of $\O$ into sets of finite perimeter $E'_1, \ldots, E_N'$, with $E_i':=E_i$ for $i=1,\ldots,N-1$ and $E'_N:=\O \setminus \bigcup_{i}^{N-1} E_i$, we can invoke \cite[Corollary 2.5]{BCG} to deduce the existence of a partition of $\O$ into polyhedral sets $\widehat E^j_1, \ldots, \widehat E^j_N$ such that, for every $i=1, \ldots, N$, 
$$
\mathcal L^n(\widehat E_i^j \triangle E'_i ) \to 0 \quad \text{and} \quad \mathcal H^{n-1} (\partial^*\widehat E_i^j ) \to \mathcal H^{n-1} (\partial^*E'_i ),  
$$ 
as $j \to +\infty$. Eventually, the desired sequence $(u_j)$ satisfying \eqref{density}-\eqref{approx} can be obtained by a standard diagonal argument.  

\medskip 

We now construct a recovery sequence $(u_\e,v_\e) \subset W^{1,2}(\O,\R^n) \times W^{1,2}(\O)$ for $F_\e$ when $u$ is as in \eqref{approx}. Therefore we have that, in particular, up to a set of zero $\mathcal H^{n-1}$-measure
\begin{equation}\label{e:jump-u}
J_u = \bigcup_{i=1}^M S_i \cap \O,  
\end{equation}
where $S_1, \ldots, S_M \subset \R^n$ are a finite number of $(n-1)$-dimensional simplexes. 

For every $i\in \{1,\ldots, M\}$ we denote with $\Pi_i$ the $(n-1)$-dimensional hyperplane containing the simplex $S_i$; we have that $\Pi_i \neq \Pi_\ell $, for $i\neq \ell$. 

We start by constructing $v_\e$. To this end, we recall that 
\begin{equation}\label{e:double-min}
C_V = 2 \int_0^1\sqrt{V(s)}\ds= \inf_{T>0} \min \bigg\{\int_0^T (V(w)+|w'|^2)\dt \colon w \in \mathcal A(0,T)\bigg\}; 
\end{equation}
where
\[
\mathcal A(0,T):=\{w\in W^{1,\infty}(0,T)\colon 0\leq w\leq 1,\,  w(0)=0,\, w(T)=1\},
\] 
(see e.g.\ \cite[Remark 6.1]{Braides})
Hence for every fixed $\eta >0$ there exists $T_\eta >0$ and $w_\eta\in W^{1,\infty}(0,T_\eta)$ with $0\leq w_\eta \leq 1$, $w_\eta(0)=0$ and $w_\eta(T_\eta)=1$ such that 
\begin{equation}\label{e:w-eta}
\int_0^{T_\eta} \big(V(w_\eta)+|w_\eta'|^2\big)\dt \leq C_V +\eta. 
\end{equation}
Let now $0<\xi_\e \ll \e$ and define the function $h_\e \colon [0,+\infty) \to [0,1]$ as
\begin{equation}\label{e:h-e}
h_\e(t):=\begin{cases} 
0 & \text{if }\; 0 \leq t \leq \xi_\e,
\cr
w_\eta \Big(\frac{t-\xi_\e}{\e}\Big) & \text{if }\; \xi_\e \leq t \leq \xi_\e + \e T_\eta,
\cr
1 & \text{if }\; t \geq \xi_\e + \e T_\eta.
\end{cases}
\end{equation}
Let $\pi_i \colon \R^n \to \Pi_i$ denote the orthogonal projection onto $\Pi_i$ and set $d_i(x):=\dist(x,\Pi_i)$, we notice that 
\[
\nabla d_i (x)=\frac{x -\pi_i(x)}{|x -\pi_i(x)|},
\]
for every $x\in \R^n \setminus \Pi_i$. 

Moreover, for every $\d>0$ we define 
\[
S_i^\delta :=\{y \in \Pi_i \colon \dist(y, S_i)\leq \d\}. 
\]
Now let $\gamma_\e^i$ be a cut-off function between $S_i^\e$ and $S_i^{2\e}$\ie $\gamma_\e^i \in C^\infty_0(S_i^{2\e})$, $0 \leq \gamma_\e^i \leq 1$, $\gamma_\e^i \equiv 1$ in $S_i^\e$, and $|\nabla \gamma_\e^i| \leq c/\e$ in $\Pi_i$, for some $c>0$. For every $i \in \{1,\ldots, M\}$ set
\begin{equation}\label{e:v-e-i}
v_\e^i(x):=\gamma_\e^i(\pi_i(x))h_\e(d_i(x))+1-\gamma_\e^i(\pi_i(x)).
\end{equation}
From the very definition of $v_\e^i$ we have that $0\leq v_\e^i \leq 1$ and $v_\e^i \in W^{1,\infty}(\R^n)$. Moreover, by using the following facts: $|\nabla \gamma_\e^i| \leq c/\e$, $\pi_i$ is Lipschitz with constant $1$, and $|\nabla d_i|=1$, we also get  
\begin{equation}\label{e:L-infty-bd}
\|\nabla v^i_\e\|_{L^\infty(\R^n)} \leq \frac{c}{\e}. 
\end{equation}
Additionally, by definition, $v_\e^i \to 1$ in $L^1_{\rm loc}(\R^n)$, and 
\begin{equation}\label{e:v-e-i-i}
v_\e^i \equiv 0\; \text{ in }\; A_i^\e\; \text{ and } \; v_\e^i \equiv 1\; \text{ in }\; \R^n \setminus B_i^\e,
\end{equation}
where 
\[
A_i^\e:=\{x \in \R^n \colon \pi_i(x)\in S_i^\e \text{ and } d_i(x)\leq \xi_\e\}
\]
and 
\[
B_i^\e:=\{x \in \R^n \colon \pi_i(x)\in S_i^{2\e} \text{ and } d_i(x)\leq \xi_\e+ \e T_\eta\}.
\]
Therefore, in view of \eqref{e:v-e-i-i} we have 
\begin{equation}\label{e:en-v-i}
\int_{B_i^\e} \bigg(\frac{V(v_\e^i)}{\e}+\e |\nabla v_\e^i|^2 \bigg)\dx = \frac{V(0)}{\e}\, \mathcal L^n(A_i^\e)
+ \int_{B_i^\e \setminus A_i^\e} \bigg(\frac{V(v_\e^i)}{\e}+\e |\nabla v_\e^i|^2 \bigg)\dx
\end{equation}
moreover, we notice that 
\begin{equation}\label{e:en-v-0}
\lim_{\e \to 0}\frac{V(0)}{\e}\, \mathcal L^n(A_i^\e)= 2 V(0) \, \lim_{\e \to 0}\frac{\xi_\e}{\e}\,\mathcal H^{n-1}(S_i^\e)=0,
\end{equation}
where to establish the last equality we have used that $\xi_\e \ll \e$ and $\mathcal H^{n-1}(S_i^\e) \to \mathcal H^{n-1}(S_i)$, as $\e \to 0$. We now estimate the second term in the right-hand side of \eqref{e:en-v-i}. To do so it is convenient to write 
\[
B_i^\e \setminus A_i^\e=H_i^\e \cup I_i^\e
\] 
where
\begin{equation}\label{e:H-i}
H_i^\e:=\{x\in \R^n \colon \pi_i(x)\in S_i^\e \text{ and } \xi_\e \leq d_i(x)\leq \xi_\e +\e T_\eta\}
\end{equation}
and 
\begin{equation}\label{e:I-i}
I_i^\e:=\{x\in \R^n \colon \pi_i(x)\in S_i^{2\e} \setminus S_i^\e \text{ and } d_i(x)\leq \xi_\e +\e T_\eta\}.
\end{equation}
By definition of $v_\e^i$, in the set $H_i^\e$ it holds 
\[
v_\e^i(x)=w_\eta \bigg(\frac{d_i(x)-\xi_\e}{\e}\bigg)
\]
and thus 
\[
\nabla v_\e^i(x) = \frac{1}{\e}w_\eta ' \bigg(\frac{d_i(x)-\xi_\e}{\e}\bigg) \nabla d_i(x). 
\] 
Therefore, since $|\nabla d_i|=1$ a.e.,\ we have 
\begin{eqnarray}\nonumber 
\int_{H_i^\e} \bigg(\frac{V(v_\e^i)}{\e}+\e |\nabla v_\e^i|^2 \bigg)\dx &=& \int_{H_i^\e} \bigg(\frac{1}{\e}V\Big(w_\eta \bigg(\frac{d_i(x)-\xi_\e}{\e}\bigg)\bigg)+\e \bigg|\frac{1}{\e}w_\eta ' \bigg(\frac{d_i(x)-\xi_\e}{\e}\bigg) \nabla d_i(x)\bigg|^2 \bigg)\dx
\\\nonumber
&=& 2 \int_{S_i^\e} \,d\mathcal H^{n-1} \int_{\xi_\e}^{\xi_\e +\e T_\eta} \bigg(\frac{1}{\e}V\Big(w_\eta \Big(\frac{t-\xi_\e}{\e}\Big)\bigg)+\frac{1}{\e} \Big|w_\eta ' \Big(\frac{t-\xi_\e}{\e}\Big)\Big|^2 \bigg)\dt
\\\nonumber 
&=& 2 \int_{S_i^\e} \,d\mathcal H^{n-1} \int_{0}^{T_\eta} \big(V(w_\eta(t))+|w_\eta ' (t)|^2 \big)\dt
\\\label{e:est-en-v-i-1}
&\leq & 2 (C_V+\eta)\, \mathcal H^{n-1} (S_i) + o(1),
\end{eqnarray}
as $\e\to 0$, where to establish the last inequality we have used \eqref{e:w-eta}. 

Furthermore, from \eqref{e:L-infty-bd} it is immediate to show that
\begin{equation}\label{e:est-en-v-i-2}
\lim_{\e \to 0}\int_{I_i^\e} \bigg(\frac{V(v_\e^i)}{\e}+\e |\nabla v_\e^i|^2 \bigg)\dx \leq \lim_{\e \to 0} \frac{c}{\e}\, \e \, \mathcal H^{n-1}(S_i^{2\e} \setminus S_i^\e)=0.  
\end{equation}
Eventually, gathering \eqref{e:en-v-i}-\eqref{e:est-en-v-i-2} yields
\begin{equation*}
\int_{B_i^\e} \bigg(\frac{V(v_\e^i)}{\e}+\e |\nabla v_\e^i|^2 \bigg)\dx \leq 2 (C_V+\eta)\, \mathcal H^{n-1} (S_i) + o(1),
\end{equation*}
as $\e \to 0$. 

Now the idea is to combine together the sequences $(v_\e^i)$ in order to define a new sequence $(v_\e)$ which belongs to $W^{1,2}(\O)$ 
and in every $B_i^\e$ coincides with $(v_\e^i)$, up to a set where the surface energy is negligible. Moreover the sequence $(v_\e)$ shall satisfy: 
$v_\e \to 1$ in $L^1(\O)$ and 
\begin{equation}\label{e:claim-v-e}
\limsup_{\e \to 0}\int_{\O} \bigg(\frac{V(v_\e^i)}{\e}+\e |\nabla v_\e^i|^2 \bigg)\dx \leq 2 (C_V+\eta)\, \mathcal H^{n-1} (J_u),
\end{equation}
where $u$ is as in \eqref{approx}. 

To this end, we define 
\begin{equation}\label{e:v-e}
v_\e:=\min\{v_\e^1,\ldots, v_\e^M\};
\end{equation}
clearly, $0\leq v_\e \leq 1$, $(v_\e) \subset W^{1,2}(\O)$, and $v_\e \to 1$ in $L^1(\O)$ hence, in particular, $v_\e \to 1$ in $L^1(\O)$. Further, setting 
\[
A^\e:= \bigcup_{i=1}^M A_i^\e\; \text{ and }\; B^\e:= \bigcup_{i=1}^M B_i^\e,
\] 
by \eqref{e:v-e-i-i} and \eqref{e:v-e} we readily deduce that
\begin{equation}\label{e:v-e-c}
v_\e \equiv 0\; \text{ in }\; A^\e\; \text{ and } \; v_\e \equiv 1\; \text{ in }\; \R^n \setminus B^\e.
\end{equation}
Then, writing $\O= (\O \setminus B^\e) \cup (\O \cap (B^\e \setminus A^\e)) \cup (\O \cap A^\e)$, in view of \eqref{e:v-e-c} we get
\begin{eqnarray}\nonumber
\int_{\O} \bigg(\frac{V(v_\e^i)}{\e}+\e |\nabla v_\e^i|^2 \bigg)\dx &\leq& \int_{\O \cap (B^\e \setminus A^\e)} \bigg(\frac{V(v_\e^i)}{\e}+\e |\nabla v_\e^i|^2 \bigg)\dx\\\label{e:resto}
&+& 2 V(0) \, \frac{\xi_\e}{\e}\,\sum_{i=1}^M\mathcal H^{n-1}(S_i^\e).
\end{eqnarray}
Since $\xi_\e \ll \e$ and 
\[
\sum_{i=1}^M\mathcal H^{n-1}(S_i^\e)\; \to \; \sum_{i=1}^M\mathcal H^{n-1}(S_i)= \mathcal H^{n-1}(J_u),
\] 
as $\e \to 0$, the second term in the right hand side of \eqref{e:resto} is negligible. Hence, to get \eqref{e:claim-v-e} we are left to estimate the surface energy in $\O \cap (B^\e \setminus A^\e)$. We claim that
\begin{equation}\label{e:claim-S}
\limsup_{\e \to 0} \mathcal S_\e:= \limsup_{\e \to 0} \int_{\O \cap (B^\e \setminus A^\e)} \bigg(\frac{V(v_\e)}{\e}+\e |\nabla v_\e|^2 \bigg)\dx \leq 2 (C_V+\eta)\, \mathcal H^{n-1} (J_u). 
\end{equation}
We notice that 
\[
B^\e \setminus A^\e = \bigcup_{i=1}^M (H_i^\e \cup I_i^\e)\setminus A^\e,
\]
where the sets $H_i^\e$ and $I_i^\e$ are defined as in \eqref{e:H-i} and \eqref{e:I-i}, respectively.  Since moreover
\[
(H_i^\e \cup I_i^\e)\setminus A^\e \subset \bigcup_{j\neq i} \Big((H_i^\e \cup I_i^\e)\cap (H_j^\e \cup I_j^\e) \Big)\cup \bigcap_{j \neq i} \Big((H_i^\e \cup I_i^\e) \setminus B_j^\e \big),
\]
we have 
\begin{eqnarray}\nonumber 
\mathcal S_\e &\leq& \sum_{i=1}^M \int_{\O \cap ((H_i^\e \cup I_i^\e)\setminus A^\e)} \bigg(\frac{V(v_\e)}{\e}+\e |\nabla v_\e|^2 \bigg)\dx
\\\nonumber 
&\leq&  \sum_{i=1}^M \int_{\O \cap \bigcap_{j \neq i} \big((H_i^\e \cup I_i^\e) \setminus B_j^\e) \big)} \bigg(\frac{V(v_\e)}{\e}+\e |\nabla v_\e|^2 \bigg)\dx
\\\nonumber 
&+& \sum_{i=1}^M \int_{\O \cap \bigcup_{j\neq i} \big((H_i^\e \cup I_i^\e)\cap (H_j^\e \cup I_j^\e) \big)} \bigg(\frac{V(v_\e)}{\e}+\e |\nabla v_\e|^2 \bigg)\dx
\\\label{e:S-1-S-2}
&=: & \mathcal S_\e^1 +\mathcal S_\e^2.
\end{eqnarray}
We now estimate the terms $\mathcal S_\e^1$  and $\mathcal S_\e^2$ separately. To this end, we start observing that 
\[
\bigcap_{j \neq i} \Big((H_i^\e \cup I_i^\e) \setminus B_j^\e \big) \subset \bigcap_{j \neq i}\{x \in \R^n \colon v_\e^i (x) \leq v_\e^j (x)\}, 
\]
hence, invoking \eqref{e:est-en-v-i-1} and \eqref{e:est-en-v-i-2}, we readily get
\begin{eqnarray}\nonumber
\mathcal S_\e^1 &=& \sum_{i=1}^M \int_{\O \cap \bigcap_{j \neq i} \big((H_i^\e \cup I_i^\e) \setminus B_j^\e) \big)} \bigg(\frac{V(v^i_\e)}{\e}+\e |\nabla v^i_\e|^2 \bigg)\dx 
\\\nonumber
&\leq& (2C_V +\eta)\sum_{i=1}^M \mathcal H^{n-1}(S_i) + o(1)
\\\label{e:S-1}
&=&(2C_V +\eta)\, \mathcal H^{n-1}(J_u) + o(1),
\end{eqnarray}
as $\e \to 0$.
Moreover, appealing to \eqref{e:L-infty-bd} easily gives
\begin{equation}\label{e:estimate-S2}
\mathcal S_\e^2 \leq \sum_{i=1}^M \sum_{j\neq i} \frac{c}{\e}\, \mathcal L^n\big((H_i^\e \cup I_i^\e)\cap (H_j^\e \cup I_j^\e) \big). 
\end{equation}
We now claim that 
\begin{equation}\label{e:van-meas}
\lim_{\e \to 0} \frac{c}{\e}\, \mathcal L^n\big((H_i^\e \cup I_i^\e)\cap (H_j^\e \cup I_j^\e) \big)=0,
\end{equation}
for every $i,j \in \{1, \ldots, M\}$. Indeed, since $\Pi_i\neq \Pi_j$ then the set $S_i \cap S_j$ is contained in an $(n-2)$-dimensional affine subspace of $\R^n$, so that by \eqref{e:H-i} and \eqref{e:I-i} we can deduce that   
\begin{equation}\label{e:van-meas-1}
\mathcal L^n\big((H_i^\e \cup I_i^\e)\cap (H_j^\e \cup I_j^\e) \big) \leq c (\xi_\e +\e T_\eta)^2= O(\e^2),
\end{equation}
as $\e \to 0$, where the constant $c>0$ depends only on the angle between $\Pi_i$ and $\Pi_j$ and on $\mathcal H^{n-2}(S_i \cap S_j)$.
Hence, \eqref{e:van-meas-1} immediately yields \eqref{e:van-meas}. 

Finally, gathering \eqref{e:S-1-S-2} and \eqref{e:S-1} entails \eqref{e:claim-S}, as desired. 

Therefore, to conclude the proof of the upper bound we now have to exhibit a sequence $(u_\e) \subset W^{1,2}(\O,\R^n)$ such that $u_\e \to u$ in $L^1(\O,\R^n)$ and  
\begin{equation}\label{e:ls-volume}
\limsup_{\e \to 0} \int_{\O} k_\e \Phi(v_\e)\, W(x,\nabla u_\e)\dx =0.
\end{equation}
To this end, set 
\[
(A^\e)':= \bigcup_{i=1}^M \Big\{x \in \R^n \colon \pi_i(x)\in S_i^{\e/2} \text{ and } d_i(x)\leq \frac{\xi_\e}{2}\Big\},
\]
let $\varphi_\e \in C^\infty_0(A^\e)$ be a cut-off function between $(A^\e)'$ and $A^\e$, and define
\[
u_\e:=(1-\varphi_\e)u. 
\]
Then, clearly $(u_\e)\subset W^{1,\infty}(\O,\R^n)$, moreover $u_\e \to u$ in $L^1(\O,\R^n)$. Moreover, since $v_\e \equiv 0$ in $A^\e$, and $\Phi$ vanishes at zero, it holds
\[
\int_{\O} k_\e \Phi(v_\e)\, W(x,\nabla u_\e)\dx= \int_{\O\setminus A^\e} k_\e \Phi(v_\e)\, W(x,\nabla u)\dx
\]
hence using that $u\in PR(\O)$ together with the fact that for every $x\in \O$ the function $W(x,\cdot)$ vanishes in $SO(n)$ we immediately get
\[
\Phi(v_\e)\, W(x,\nabla u)=0 \; \text{ a.e.\ in $\O \setminus A^\e$}
\]
and hence the claim. 
\end{proof}

\begin{remark}[Approximation of inhomogeneous anisotropic perimeter functionals]\label{r:inhomogeneous anisotropic}

{\rm Arguing \\ as in the proof of \cite{Vic} (see also \cite[Theorem 3.1]{focardi}),
in view of Proposition \ref{p:domain} one can establish a $\Gamma$-convergence result for functionals of the form
\begin{equation}\label{Fe-an}
F^\phi_\e(u,v):=\begin{cases}
\ds\int_{\O}\bigg(k_\e \Phi(v)\,W(x,\nabla u)+\frac{V(v)}{\e}
+\e\phi^2(x,\nabla v)\bigg)\dx & \hspace{-2mm} (u,v)\in W^{1,2}(\O,\R^n)\times W^{1,2}(\O),
\cr 
 & 0\leq v\leq 1 \; \text{ a.e. on $\O$,}\; 
\cr
+\infty & \text{otherwise,}
\end{cases}
\end{equation} 
where the euclidean norm in $F_\e$ is now replaced by a Finsler norm $\phi$. That is, $\phi \colon \O \times \R^n \to [0,+\infty)$ is a continuous function which is convex in its second variable and satisfies the two following properties:  
\begin{itemize}
\item[i.] for every $(x,z)\in \O\times \R^n$ and for every $t\in \R$
\[
\phi(x,tz)=|t|\phi(x,z);
\]
\item[ii.] for every $(x,z)\in \O\times \R^n$ there exist $0<m\leq M<+\infty$ such that
\[
m |z| \leq \phi(x,z) \leq M |z|.
\]
\end{itemize}
In this case it can be proven that the family of functionals $(F^\phi_\e)$ $\Gamma(L^1(\O,\R^n)\times L^1(\O))$-converges 
to the following inhomogeneous and anisotropic functional $F^\phi \colon L^1(\O,\R^n)\times L^1(\O)\longrightarrow [0,+\infty]$ defined on piecewise rigid maps as:
\begin{equation*}
F^\phi(u,v):=\begin{cases}
\ds 2C_V\int_{J_u} \phi(x,\nu_u)\,d\mathcal H^{n-1} & \text{if $u \in PR(\O)$\; and }\; v=1\; \text{a.e. in }\; \O,
\cr
+\infty & \text{otherwise},
\end{cases}
\end{equation*}
where $\nu_u$ denotes the exterior unit normal to $J_u$. 
}
\end{remark}

\section{Incompatible wells and linearised elasticity}\label{s:generalizations}

In this section we are going to address two possible extensions of Theorem \ref{thm:G-conv}. We first discuss a generalisation of Theorem \ref{thm:G-conv} to the case where the zeros of the potential $W$ lie in a suitable nonempty compact set $\mathcal K$. Then, we show that our proof-strategy also applies to the case of linearised elasticity. 
Similarly as in Section \ref{s:main}, also in these cases the key tools for the analysis are two suitable variants of 
the piecewise-rigidity property stated in Theorem \ref{t:CGP} (cf.\ Theorem \ref{t:CGPK} and Theorem \ref{t:CGP-sim}). 

\subsection{The case of $\mathcal K$ piecewise-rigid maps}\label{ss:K rigidity}

Let $U \subset \R^n$ be a bounded domain with Lipschitz boundary, and let 
$\mathcal K \subset \MM^{n \times n}$ be a nonempty compact set satisfying the following $L^p$-quantitative rigidity estimate for some $p \in (1,n/(n-1))$: there exists a constant $C>0$ (depending only on $p$ and $n$) such that for every $u\in W^{1,p}(U, \R^{n})$
\begin{equation}\label{quant-rig}
\min_{\AAA\in \mathcal K} \|\nabla u -\AAA\|_{L^p(U, \MM^{n\times n})} \leq C \|\dist (\nabla u,\mathcal K)\|_{L^p(U)}.
\end{equation}
We notice that \eqref{quant-rig} implies the rigidity of the differential inclusion
\begin{equation}\label{e:diff incl}
v\in W^{1,\infty}(U,\R^n)\, \mbox{ and }\,
\nabla v(x)\in \mathcal K\quad \text{a.e. $U$},
\end{equation}
in the sense explained in Lemma \ref{l:rigidity} below. In the statement of Lemma \ref{l:rigidity} we use the same terminology adopted in \cite[Chapter 8]{Rindler}
(see also \cite[Section 1.4]{MLN} and \cite{Kirch}).
\begin{lemma}\label{l:rigidity}
Let $U \subset \R^n$ be a bounded domain with Lipschitz boundary. Let $\mathcal K \subset \MM^{n \times n}$ be a nonempty compact set satisfying \eqref{quant-rig}. Then, the following statements hold true: 
\begin{enumerate}
\item the differential inclusion \eqref{e:diff incl} is rigid for exact solutions\ie
the only solutions to \eqref{e:diff incl} are affine functions;

\smallskip

\item the differential inclusion \eqref{e:diff incl} is rigid for approximate solutions\ie if $\dist(\nabla u_j,\mathcal K)\to 0$ in measure in $U$, $(u_j)$ converges to 
$u$ weakly* in $W^{1,\infty}(U,\R^n)$, $u_j=\AAA x$ on $\partial U$ for some \RRR \EEE $\AAA \in \mathbb M^{n\times n}$, then 
$(\nabla u_j)$ converges in measure to $\nabla u$ in $U$ and $u$ is affine; 

\smallskip

\item the differential inclusion \eqref{e:diff incl} is strongly rigid\ie if $\dist(\nabla u_j,\mathcal K)\to 0$ in measure in $U$ and $(u_j)$ converges to $u$ weakly* in $W^{1,\infty}(U,\R^n)$, then 
$(\nabla u_j)$ converges in measure to $\nabla u$ in $U$ and $u$ is affine;

\smallskip

\item we have
\begin{equation}\label{qc-hull}
\mathcal{K}=\mathcal K^{\rm qc}\,,
\end{equation}
\end{enumerate}
where $\mathcal K^{\rm qc}$ denotes the quasiconvex envelope of $\mathcal K$\ie
\[
\mathcal K^{\rm qc}:=\{\AAA\in\MM^{n\times n}:\, f(\AAA)\leq\sup_{\mathcal K}f,\; \forall \mbox{$f:\MM^{n\times n}\to\R$ quasiconvex}\}\,.
\]
\end{lemma}
For the readers' convenience the proof of Lemma \ref{l:rigidity} is included in the Appendix \ref{a:proof rigidity lemma}. 

\medskip

Below we give a list of nonempty compact sets $\mathcal K\subset\MM^{n\times n}$ 
for which \eqref{quant-rig} holds true. The most prominent examples are due to Ball and James \cite[Proposition 2]{BJ1987} and to Friesecke, James, and M\"uller \cite[Theorem 3.1]{FJM} and correspond, respectively, to the case of two non rank-1 connected matrices and to that of $SO(n)$. 

We notice that in the examples (1) and (3) below, property \eqref{quant-rig} directly follows from an incompatibility condition for the approximate solutions of \eqref{e:diff incl}, as shown in \cite{ChM} (see also \cite[Theorem 1.2]{DLS}). This condition reduces rigidity for multiple-wells to a single-well rigidity statement. We recall here that two disjoint compact sets $K_{1}, K_{2}\in \MM^{n \times n}$ are incompatible for the differential inclusion \eqref{e:diff incl}, with $\mathcal K= K_{1}\cup K_{2}$, if for any sequence $(u_{j})\subset W^{1,\infty}(U, \R^n)$ such that $\dist(\nabla u_j, K_{1}\cup K_{2})\to 0$ in measure, then either $\dist(\nabla u_j, K_{1})\to 0$ or $\dist(\nabla u_j, K_{2})\to 0$ in measure. In this case $K_{1}$ and $K_{2}$ are also called incompatible energy-wells.

In the examples (4) and (5) listed below, property \eqref{quant-rig} is instead a consequence of the Friesecke, James, and M\"uller rigidity estimate \cite[Theorem 3.1]{FJM} for (2), and of the above mentioned incompatibility for approximate solutions of \eqref{e:diff incl}.  
Although equality \eqref{qc-hull} is a consequence of \eqref{quant-rig} as established by Lemma \ref{l:rigidity}, for each example in the list below we also give a precise reference to a direct proof of \eqref{qc-hull}. We refer the reader to \cite{MLN}, \cite{Kirch} and \cite[Chapter 8]{Rindler} for more details on these topics. 
\begin{enumerate}

\item $\mathcal K=\{\AAA_{1},\AAA_{2}\}$, where $\AAA_{1},\AAA_{2}\in \MM^{n\times n}$ are not rank-1 connected, see \cite[Proposition 2]{BJ1987}, (see also \cite[Example 4.3]{Zh92});
\smallskip

\item $\mathcal K=SO(n)$ \cite[Theorem 3.1]{FJM} (for \eqref{qc-hull} see \cite{Kind} and also \cite[Example 4.4]{Zh92}); 



\item $\mathcal K=\{\AAA_{1},\AAA_{2},\AAA_{3}\}$, where $\AAA_{1},\AAA_{2},\AAA_{3}\in \R^{n\times n}$ are such that $\mathcal K$ has no rank-1 connections, see \cite[Section 4]{SvARMA1992};

\smallskip

\item $\mathcal K=\bigcup_{i=1}^{N}\AAA_{i}SO(2)$, where $\AAA_{i}\in \R^{2\times 2}$ are such that ${\rm det}\AAA_{i}>0$ for all $i\in\{1,2,3\}$ and $\mathcal K$ has no rank-1 connections, see \cite[Theorem 2 and Remark 1]{SvIMA1992};

\smallskip

\item $\mathcal K=SO(3)\cup SO(3)\mathbb H$, where $\mathbb H={\rm diag}(h_{1},h_{2},h_{3})$, $h_{1}\geq h_{2}\geq h_{3}>0$ and $h_{2}\neq 1$ (the latter condition is equivalent to $\mathcal K$ having no rank-1 connections). Additionally, one of the following two conditions must hold true: 
\begin{enumerate}
\item[(i)] there exists $i$ such that $(h_{i}-1)(h_{i-1}h_{i+1}-1)\geq 0$ (here the indices are counted modulo $3$); 

\item[(ii)] $h_{1}\geq h_{2}>1>h_{3}>\frac13$ or $3>h_{1}>1> h_{2}\geq h_{3}>0$;   
\end{enumerate}
see \cite[Theorem 1.2]{DKMS}.
\end{enumerate}

\medskip

We now recall an extension of the piecewise-rigidity result contained 
in Theorem \ref{t:CGP} to the case of a compact set $\mathcal K$ for which \eqref{quant-rig} holds true. We state this result for $GSBV$-functions and we refer the reader to \cite[Theorem 2.1]{CGP} for the original statement in the $SBV$-setting. 

\begin{theorem}\label{t:CGPK}
Let $\mathcal K \subset \MM^{n \times n}$ be a nonempty compact set for which \eqref{quant-rig} holds true and let $u\in GSBV(\O,\R^n)$ be such that
$\mathcal{H}^{n-1}(J_u)<+\infty$ and $\nabla u \in \mathcal K$ a.e. in $\O$. Then, $u \in PR_{\mathcal K}(\O)$\ie 
\begin{equation}\label{C-a-K}
u(x)=\sum_{i\in \N} (\AAA_i x+ b_i)\chi_{E_i}(x),
\end{equation}
with $\AAA_i\in \mathcal K$ for every $i \in\N$, $b_i \in \R^n$, and $(E_i)$ Caccioppoli partition of $\O$.  
\end{theorem} 
 
For $\e>0$ we consider the functionals $F^{\mathcal K}_\e \colon L^1(\O,\R^{{n}})\times L^1(\O) \longrightarrow [0,+\infty]$ defined as 
\begin{equation}\label{Fe-K}
F^{\mathcal K}_\e(u,v):=\begin{cases}
\ds \int_{\O}\bigg(k_\e \Phi(v)\,W(x,\nabla u)+\frac{V(v)}{\e}+\e|\nabla v|^2\bigg)\dx & (u,v)\in W^{1,2}(\O,\R^{n})\times W^{1,2}(\O),
\cr 
 & 0\leq v\leq 1 \, \text{a.e. in}\; \O
\cr
+\infty & \text{otherwise,}
\end{cases}
\end{equation}
where $k_\e \to +\infty$, as $\e\to 0$, 
$W \colon \O \times \MM^{n\times n} \to [0,+\infty)$ is a Borel function such that $W(x,\AAA)=0$ for every $\AAA\in \mathcal K$. We assume moreover that for every $x\in \O$ and every $\AAA\in \MM^{n\times n}$ it holds $W(x,\AAA)\geq \alpha \, \dist^2(\AAA,\mathcal K)$, for some $\alpha >0$.
%
%
%
%

By combining Proposition \ref{p:domain-K} and Theorem \ref{thm:G-conv-K} below we can identify the $\G$-limit of $F^\mathcal K_\e$. Now using Theorem \ref{t:CGPK} in place of Theorem \ref{t:CGP}, these results can be proven by following exactly the same arguments employed  in the proofs of Proposition \ref{p:domain} and Theorem \ref{thm:G-conv}, respectively. 
For the readers' convenience we notice that in the proof of Proposition \ref{p:domain-K} the following characterisation of $\mathcal K^{\rm qc}$ is needed: 
\[
\mathcal K^{\rm qc}=\{\AAA \in \MM^{n\times n} \colon Q(\dist^q(\cdot,\mathcal K))(\AAA)=0\},
\]
for every $q\in[1,+\infty)$ (cf. \cite[Proposition 2.14]{Zh92}, and also \cite[Theorem 4.10]{MLN}) together with Lemma \ref{l:rigidity}, statement (4). 

Instead, the proof of Theorem \ref{thm:G-conv-K} is identical to that of Theorem \ref{thm:G-conv} up to replacing the coercivity estimate \eqref{e:dist coerc} with 
\[
\dist^2(\AAA,\mathcal K)\geq \left[\frac12|\AAA\xi|^2-\max_{\mathcal K}|\mathbb{K}|^2\right]_+\,,
\]
which holds true for every $\AAA\in\MM^{n\times n}$ and every $\xi\in\mathbb{S}^{n-1}$. 

The following proposition shows that the $\G$-limit of $(F^{\mathcal K}_\e)$ (if it exists) is finite only on $PR_{\mathcal K}(\O)$.
\begin{proposition}\label{p:domain-K}
 Let $(u_\e,v_\e) \subset W^{1,2}(\O,\R^{n})\times W^{1,2}(\O)$,  $0\leq v_\e \leq 1$ a.e.\ in $\O$, be such that 
 $$
 \liminf_{\e \to 0}F^{\mathcal K}_\e(u_\e,v_\e)<+\infty 
 \quad \text{and}\quad u_\e \to u \; \text{ in } \; L^1(\O,\R^n).
 $$
 Then, $v_\e \to 1$ in $L^1(\O)$ and $u \in PR_\mathcal{K}(\O)$.
 \end{proposition}
In addition, the following $\Gamma$-convergence result holds true.

\begin{theorem}
\label{thm:G-conv-K} 
The family $(F^{\mathcal K}_\e)$ defined in \eqref{Fe-K} $\Gamma(L^1(\O,\R^{n})\times L^1(\O))$-converges to the functional 
$F^\mathcal{K} \colon L^1(\O,\R^{n})\times L^1(\O)\longrightarrow [0,+\infty]$ given by
\begin{equation}\label{F-K}
F^{\mathcal K}(u,v):=\begin{cases}
2C_V \mathcal H^{n-1}(J_u) & \text{if $u \in PR_{\mathcal K}(\O)$}, \; v=1\; \text{ a.e. in }\; \O
\cr
+\infty & \text{otherwise},
\end{cases}
\end{equation}
where $C_V:=2\int_0^1 \sqrt{V(s)}\,ds$. 
\end{theorem}

\subsection{The case of linearised elasticity}
An approximation result similar to that proven in Theorems \ref{thm:G-conv} and \ref{thm:G-conv-K} can be established for interfacial energies appearing in the context of linearised elasticity. In this case the energy functionals are defined on piecewise-rigid displacements\ie on displacements of a linearly elastic body which does not store elastic energy.
\begin{definition}
A map $u \colon \O \to \R^n$ is called a piecewise-rigid displacement if there exist skew-symmetric matrices $\AAA_i\in \MM^{n\times n}_{\rm skew}$ and vectors $b_i\in\R^n$ such that 
\begin{equation}\label{PR-D}
u(x)=\sum_{i\in \N} (\AAA_i x+ b_i)\chi_{E_i}(x),
\end{equation}
with $(E_i)$ Caccioppoli partition of $\O$.  
The set of piecewise-rigid displacements on $\O$ will be denoted by $PRD(\O)$. 
\end{definition}
The next piecewise-rigidity result corresponds to Theorems \ref{t:CGP} and \ref{t:CGPK} in the setting of linearised elasticity. We state it here for $GSBD$ maps, the original version \cite[Theorem A.1]{CGP} being stated in the $SBD$-setting. The proof of this generalization is implied for instance by
\cite[Theorem 2.1]{Fr} (see also the related comments in \cite[Remark 2.2 $(i)$]{Fr}). 

In what follows $e(u)$ denotes the symmetrized approximate gradient of $u\in GSBD(\O)$ (cf. 
\cite{DM}).  

\begin{theorem}\label{t:CGP-sim}
Let $u\in GSBD(\O)$ be such that 
$\mathcal{H}^{n-1}(J_u)<+\infty$ and $e(u)=0$ a.e. in $\O$. Then, $u \in PRD(\O)$. 
\end{theorem}

To approximate interfacial energies defined on $PRD(\O)$, for $\e>0$, we consider the functionals $E_\e \colon L^1(\O,\R^{{n}})\times L^1(\O) \longrightarrow [0,+\infty]$ defined as 
\begin{equation}\label{E-e-lin}
E_\e(u,v):=\begin{cases}
\ds \int_{\O}\hspace{-1mm}\bigg(\hspace{-1mm}k_\e \Phi(v)\,
W(x,\nabla u) \dx+\frac{V(v)}{\e}+\e|\nabla v|^2\bigg)\dx \hspace{-2mm} & (u,v)\in W^{1,2}(\O,\R^n)\times W^{1,2}(\O),
\cr 
 & 0\leq v\leq 1 \, \text{a.e. in}\; \O
\cr
+\infty & \text{otherwise,}
\end{cases}
\end{equation}
where $k_\e \to +\infty$, as $\e\to 0$, 
$W:\O\times\MM^{n\times n}\to[0,+\infty)$ is a 
Borel function such that for every $x\in \O$ and for every 
$\AAA\in\MM^{n\times n}$ it holds $\alpha |\AAA^{\sym}|^2 \leq  W(x,\AAA)$ for some $\alpha>0$. 
Here we denote by $\AAA^{\sym}$ the symmetric part of $\AAA$, namely $\AAA^{\sym}:=\frac{\AAA+\AAA^T}2$. 
We use standard notation for the strain $e(u)=(\nabla u)^{\sym}$ of $u\in W^{1,2}(\O,\R^n)$. 

\begin{remark}
{\rm 
We refer the reader to \cite[Remark~4.14]{CFI17} for an explicit example of a nonconvex, polyconvex function 
which depends non-trivially on the skew-symmetric part of $\AAA$ and satisfies the bounds 
 \[
 \alpha |\AAA^{\rm sym}|^2 \leq  W(x,\AAA)\leq\beta (|\AAA^{\rm sym}|^2+1) 
 \]
for some $\alpha,\,\beta>0$, for every $x\in \O$ and for every 
$\AAA\in\MM^{n\times n}$. 

Another example can be obtained by taking $W(\AAA)=h^2(\AAA)$, $\AAA\in\MM^{n\times n}$, 
where $h$ is a one-homogeneous quasiconvex function such that for every $\AAA\in\MM^{n\times n}$ 
and for some $\alpha,\,\beta>0$
\[
\alpha|\AAA^\sym|\leq h(\AAA)\leq \beta|\AAA^\sym|\,,
\]
with $h$ depending non-trivially on $\AAA^{\rm skew}:=\frac{\AAA-\AAA^T}{2}$. We notice that, in particular, $h$ (and therefore $W$) is not convex. 
A function as above can be obtained by slightly modifying M\"uller's celebrated example 
\cite{Mu92} similarly as in \cite[Section 7]{CFVG20}. 
}
\end{remark}

In the following proposition we show that the $\G$-limit of $(E_\e)$ (if it exists) is finite only on piecewise-rigid displacements. 
\begin{proposition}\label{p:domain-lin}
Let $(u_\e,v_\e) \subset W^{1,2}(\O,\R^n)\times W^{1,2}(\O)$  
be such that 
$$
\liminf_{\e \to 0}E_\e(u_\e,v_\e)<+\infty 
\quad \text{and}\quad u_\e \to u \; \text{ in } \; L^1(\O,\R^n).
$$
Then, $v_\e \to 1$ in $L^1(\O)$ and $u \in PRD(\O)$.  
\end{proposition}

\begin{proof}
We argue as in the proof of Proposition \ref{p:domain} 
up to estimate \eqref{e:tildeueps prop} upon replacing $F_\e$ 
with $E_\e$, $\dist^2( \nabla u_\e,SO(n))$ with $|e(u_\e)|^2$,
and similarly for the corresponding quantities depending on 
$\tilde u_\e$. Therefore \eqref{dist-to-0} now becomes
\begin{equation}\label{dist-to-0 eu}
E_\e(u_\e,v_\e)\geq \alpha\,k_\e \Phi(\delta)
\int_\O|e(\tilde{u}_\e)|^2dx\,.
\end{equation}
Since by assumptions $\sup_{\e}E_\e(u_\e,v_\e)<+\infty$, recalling that $k_{\e}\to+\infty$ as $\e\to 0$, inequality \eqref{dist-to-0 eu} implies $\lim_{\e}\|e(\tilde{u}_\e)\|_{L^2(\O,\MM^{n\times n}_\sym)}=0$. The latter, together with $\sup_\e\mathcal{H}^{n-1}(J_{\tilde{u}_\e})<+\infty$
(cf. \eqref{e:salto tildeueps}), allows us to exploit Dal Maso's 
$GSBD$ Compactness Theorem \cite[Theorem~11.3]{DM} which implies both that 
$u \in GSBD(\O)$ and that
\[
e(\tilde u_\e) \wto e(u) \; \text{ in $L^2(\O,\MM^{n \times n}_{\sym})$ \; and } \; 
\mathcal H^{n-1}(J_u)\leq\liminf_{\e\to 0} \mathcal H^{n-1}(J_{\tilde u_\e})<+\infty\,
\]
Therefore, we get
\begin{equation*}
\int_{\O} |e(u)|^2\dx 
\leq \lim_{\e \to 0} \int_{\O} |e(\tilde u_\e)|^2\dx =0\,,
\end{equation*}
and thus $e(u)=0$ a.e.\,in $\O$. Theorem~\ref{t:CGP-sim} then yields
 that $u \in PRD(\O)$. 
\end{proof}
\begin{remark}\label{strong-conv}
{\rm
We notice that from the weak convergence of $e(\tilde u_\e)$ to $e(u)$ 
in $L^2(\O,\MM^{n\times n}_\sym)$ and the convergence of $\|e(\tilde u_\e)\|_{L^2(\O,\MM^{n\times n}_\sym)}$ 
to $\|e(u)\|_{L^2(\O,\MM^{n\times n}_\sym)}$, we conclude that
\begin{equation}\label{strong-sim}
e(\tilde u_\e) \to e(u)\;  \text{ in }\; L^2(\O,\MM^{n\times n}_\sym).
\end{equation}
}
\end{remark}

\noindent Arguing as in the proof of Theorem \ref{thm:G-conv}, on account of Proposition \ref{p:domain-lin} we can now prove the following $\Gamma$-convergence result for the family $(E_\e)$.

\begin{theorem}\label{thm:G-conv-lin} 
The family of functionals $(E_\e)$ defined in \eqref{E-e-lin} $\Gamma(L^1(\O,\R^n)\times L^1(\O))$-converges 
to the functional $E \colon L^1(\O,\R^n)\times L^1(\O)\longrightarrow [0,+\infty]$ given by
\begin{equation}\label{E}
E(u,v):=\begin{cases}
2C_V \mathcal H^{n-1}(J_u) & \text{if $u \in PRD(\O)$\; and }\; v=1\; \text{a.e. in }\; \O,
\cr
+\infty & \text{otherwise},
\end{cases}
\end{equation}
where $C_V:=2\int_0^1 \sqrt{V(s)}\,ds$. 
\end{theorem}

\begin{proof}
The proof is very similar to that of Theorem~\ref{thm:G-conv}, therefore we highlight only the necessary changes, referring for the notation and the details to that proof.

In order to establish the lower bound inequality in Step~1 there, we have to show \eqref{e:surfaceinq}.
Assuming that $(u_\e,v_\e)$ converges to $(u,1)$ in $L^1(\O,\R^{n+1})$ and $\sup_\e E_\e(u_\e,v_\e)<+\infty$,
we can find a subsequence $(u_{\e_j},v_{\e_j})$ of $(u_\e,v_\e)$ 
such that \eqref{numsal} is satisfied, and in place of \eqref{e:sliceconv} and \eqref{e:fette} we have 
for $\mathcal{H}^{n-1}$-a.e.\ $y\in\pixi(\Omega)$
\begin{equation*}\label{e:sliceconv 2}
\big(\langle(u_{\e_j})^\xi_y,\xi\rangle,(v_{\e_j})^\xi_y\big) \to (\langle\uxiy,\xi\rangle,1)\;
\text{ in }\; L^1(\Omega^\xi_y,\R^n) \times L^1(\Omega^\xi_y),
\end{equation*}
and 
\begin{align}\label{e:fette 2}
\liminf_{j\to+\infty}\bigg(
\int_{U^\xi_y}&\bigg(\alpha \, k_{\e_j}(v_{\e_j})^\xi_y|\langle((u_{\e_j})^\xi_y)',\xi\rangle|^2 
+\frac {V\big((v_{\e_j})^\xi_y\big)}{\e_j}
+\e_j\big(((v_{\e_j})^\xi_y)'\big)^2\bigg)\,dt\notag\\
&+\Ho(J_{\chi_{(\O_{\e_j})^\xi_y}}\cap U)\bigg)\leq c<\infty,
\end{align}
for some constant $c>0$ (which may depend on $y$). To deduce \eqref{e:fette 2} we have used that 
$|\langle\AAA\xi,\xi\rangle|^2\leq|\AAA|^2$ for every $\xi\in\mathbb S^{n-1}$ and for every 
$\AAA\in\mathbb{M}^{n\times n}_\sym$.

Noting that by linearity $\langle ((u_{\e_j})^\xi_y)',\xi\rangle=\langle (u_{\e_j})^\xi_y,\xi\rangle'$,
the rest of the argument is exactly the same as the corresponding one in Theorem~\ref{thm:G-conv} now replacing $(u_{\e_j})^\xi_y$ and 
$\uxiy$ with $\langle(u_{\e_j})^\xi_y,\xi\rangle$ and $\langle\uxiy,\xi\rangle$, respectively.
 
The proof of the upper bound inequality in Step~2 of Theorem~\ref{thm:G-conv} remains unchanged up to replacing rotation matrices with antisymmetric ones.  
\end{proof}

\appendix 
\section{Proof of the rigidity lemma}\label{a:proof rigidity lemma}
\begin{proof}[Proof of Lemma \ref{l:rigidity}]
We start by showing statement (3), from which (1) and (2) immediately follow (we notice that actually the validity of (1) 
and (2) is equivalent to (3), as shown in \cite[Corollary 8.9]{Rindler}). 
To this end let $(u_j)$ and $u$ be as in (3); then $(\dist(\nabla u_j,\mathcal K))$ converges to $0$ in $L^p(U)$, indeed it converges to $0$ in measure and 
$\nabla u_j$ is bounded in $L^\infty(U,\MM^{n\times n})$. 
Thanks to \eqref{quant-rig} we can find $\AAA_j\in\mathcal K$ such that
\[
\|\nabla u_j-\AAA_j\|_{L^p(U,\MM^{n\times n})}\leq C\|\dist(\nabla u_j,\mathcal K)\|_{L^p(U)}\,.
\]
For an arbitrary subsequence $(j_k)$, we extract a further subsequence $(j_{k_h})$ such that $\AAA_{j_{k_h}}$ converges to some $\AAA\in\mathcal K$. Therefore, $(\nabla u_{j_{k_h}})$ converges to $\AAA$ in $L^p(U,\MM^{n\times n})$. This convergence combined with the weak* convergence of $(u_j)$ to $u$ in $W^{1,\infty}(U,\R^n)$ immediately gives $\nabla u=\AAA$ a.e. on $U$.
Being the limit independent of the subsequence, the Urysohn property implies that the whole sequence $(\nabla u_j)$ converges to $\AAA$ in $L^p(U,\MM^{n\times n})$ and hence the claim.

We finally prove (4) directly from \eqref{quant-rig}, despite 
its validity is well-known in literature as a consequence of (2) 
(cf. for instance \cite[Theorem 4.10]{MLN}). To conclude 
we only need to prove that $\mathcal K^{\rm qc}\subseteq \mathcal K$. 
To do so we use that
\[
\mathcal K^{\rm qc}=\{\AAA \in \MM^{n\times n} \colon Q(\dist^q(\cdot,\mathcal K))(\AAA)=0\}
\]
for every $q\in[1,+\infty)$ (cf. \cite[Proposition 2.14]{Zh92}, and also \cite[Theorem 4.10]{MLN}). 
Let $\AAA\in\mathcal K^{\rm qc}$; then the definition of quasi-convex envelope of the distance function yields the existence of $\varphi_j\in W^{1,\infty}_0(U,\R^n)$ such that 
\[
\lim_{j\to +\infty}\int_U \dist(\AAA+\nabla\varphi_j(x),\mathcal K)dx=0.
\]
Moreover, the Zhang Lemma  (cf. \cite{Zh92}, and also \cite[Lemma 4.21]{MLN})
provides us with a sequence $(\phi_j)\subset  W^{1,\infty}_0(U,\R^n)$ such that
\[
\sup_{j\in \N}\|\nabla\phi_j\|_{L^\infty(U,\MM^{n\times n})}<+\infty,\quad
\lim_{j\to +\infty}\mathcal L^n(\{\phi_j\neq\varphi_j\})=0\,.
\]
Thus, being $\mathcal K$ compact and $(\nabla\phi_j)$ bounded in $L^\infty(U,\MM^{n\times n})$, 
we obtain
\begin{align*}
\int_U \dist^p(\AAA+\nabla\phi_j(x),\mathcal K)dx&\leq
\int_U \dist^p(\AAA+\nabla\varphi_j(x),\mathcal K)dx
+C\mathcal L^n(\{\phi_j\neq\varphi_j\})\\
&\leq C\int_U \dist(\AAA+\nabla\varphi_j(x),\mathcal K)dx
+C\mathcal L^n(\{\phi_j\neq\varphi_j\})\,,
\end{align*}
thus, eventually,  
\[
\lim_{j\to +\infty}\int_U \dist^p(\AAA+\nabla\phi_j(x),\mathcal K)dx
=0\,.
\]
For $x\in U$ let $u_j(x)=\AAA x+\phi_j(x)$. By \eqref{quant-rig}, 
there exists $\AAA_j\in\mathcal K$ such that 
\[
\|\nabla u_j -\AAA_j\|_{L^p(U, \MM^{n\times n})} \leq C \|\dist (\nabla u_j,\mathcal K)\|_{L^p(U)}\,,
\]
then by the Jensen Inequality we get 
\[
\limsup_{j \to +\infty}\mathcal L^n(U)|\AAA-\AAA_j|^p\leq\lim_{j\to +\infty}\int_U|\AAA -\AAA_j+\nabla\phi_j|^pd x=0\,,
\]
and therefore $\AAA\in\mathcal K$.
\end{proof}

\section*{Acknowledgments}
\noindent The work of M. Cicalese was supported by the DFG Collaborative Research Center TRR 109, ``Discretization in Geometry and Dynamics''.
The work of C. I. Zeppieri was supported by the Deutsche Forschungsgemeinschaft (DFG, 
German Research Foundation)  under the Germany Excellence Strategy EXC 2044-390685587, Mathematics M\"unster: Dynamics--Geometry--Structure.
M. Focardi has been partially supported by GNAMPA.

\end{document}